\documentclass[oneside,a4paper,12pt,reqno]{amsart}
\usepackage{amsfonts}
\usepackage{amssymb}
\usepackage{amsthm}
\usepackage{amsmath}
\usepackage{a4wide}
\usepackage{mathrsfs}
\usepackage{epsfig}
\usepackage{palatino}
\usepackage{tikz,fp,ifthen}
\usepackage{float}
\usepackage{graphicx}
\usepackage{enumitem}

\usepackage{color}
\definecolor{citegreen}{rgb}{0,0.6,0}
\definecolor{refred}{rgb}{0.8,0,0}
\usepackage[colorlinks, urlcolor=blue, citecolor=citegreen, linkcolor=refred]{hyperref}

\usepackage{mathtools}
\mathtoolsset{showonlyrefs}

\newcommand{\N}{\mathbb{N}}

\newcommand{\R}{\mathbb{R}}
\newcommand{\SSS}{\mathbb{S}}

\newcommand{\cH}{\mathcal{H}}

\newcommand{\cL}{\mathcal{L}}

\newcommand{\sE}{\mathscr{E}}
\newcommand{\sF}{\mathscr{F}}

\newcommand{\sH}{\mathscr{H}}

\newcommand{\Imm}{{\mathrm{Imm}}\,}
\newcommand{\id}{{\mathrm{Id}}}

\newcommand{\pa}{\partial}
\newcommand{\con}{\subseteq}
\newcommand{\na}{\nabla}
\newcommand{\ep}{\varepsilon}
\newcommand{\ga}{\gamma}

\newcommand{\al}{\alpha}
\newcommand{\de}{\delta}

\newcommand{\la}{\lambda}

\newcommand{\ro}{\rho}
\newcommand{\si}{\sigma}

\newcommand{\Om}{\Omega}
\newcommand{\vp}{ X }

\DeclarePairedDelimiter\scal{\langle}{\rangle}

\theoremstyle{plain}
\newtheorem{thm}{Theorem}[section] 

\newtheorem{prop}[thm]{Proposition}
\newtheorem{lemma}[thm]{Lemma}
\newtheorem{cor}[thm]{Corollary}

\theoremstyle{definition}
\newtheorem{defn}[thm]{Definition}

\theoremstyle{remark}

\newtheorem{remark}[thm]{Remark}

\numberwithin{equation}{section}

\title[The \L ojasiewicz--Simon inequality for the elastic flow]{The \L ojasiewicz--Simon inequality for the elastic flow}

\author{Carlo Mantegazza}
\address{Carlo Mantegazza\\
	Dipartimento di Matematica e Applicazioni,
	Universit\`a di Napoli, Via Cintia, Monte S. Angelo 80126 Napoli,
	Italy}
\email{c.mantegazza@sns.it}

\author{Marco Pozzetta}
\address{Marco Pozzetta\\
Dipartimento di Matematica, Universit\`{a} di Pisa, Largo Bruno Pontecorvo 5, 56127 Pisa, Italy}
\email{pozzetta@mail.dm.unipi.it}

\date{\today}

\begin{document}

\begin{abstract}
We define the elastic energy of smooth immersed closed curves in $\R^n$ as the sum of the length and the $L^2$--norm of the curvature, with respect to the length measure. We prove that the $L^2$--gradient flow of this energy smoothly converges asymptotically to a critical point. One of our aims was to the present the application of a \L ojasiewicz--Simon inequality, which is at the core of the proof, in a quite concise and versatile way.
\end{abstract}

\maketitle

\tableofcontents

\section{Introduction}

We consider the flow by the gradient of the ``classical'' elastic energy  associated to any regular closed curve $\ga:\SSS^1\to \R^n$
\[
\sE(\ga) = \int_{\SSS^1} \left( 1+\frac12 |k|^2\right) \,ds,
\]
where $k$ is the curvature vector of $\ga$ and $ds=|\ga'(\theta)|\,d\theta$ denotes the canonical arclength measure of $\gamma$. We remark that $\sE$ is a ``geometric'' functional, that is, its value is independent of the parametrization of the curve, moreover it is well defined for every $\gamma\in H^2(\SSS^1,\R^n)\subseteq C^1(\SSS^1,\R^n)$.

We will call $\tau=|\ga'(\theta)|^{-1}\ga'(\theta)$ the unit tangent vector of $\ga$ (which is well defined as $\gamma$ is regular, that is, $|\ga'(\theta)|\not=0$ for every $\theta\in\SSS^1$) and we will denote by $\pa_s = |\ga'(\theta)|^{-1}\pa_\theta$ the differentiation with respect to the arclength $s$ of $\gamma$ (where $\partial_\theta$ is the standard derivative with respect to $\theta\in\SSS^1$, so that $\gamma'=\partial_\theta\gamma$). Recall that then the curvature is given by $k=\pa_{ss}^2 \ga=\pa_s\tau$, which is a normal vector field along $\gamma$.\\
If $\vp:\SSS^1\to\R^n$ is a vector field along $\gamma$, we define 
$$
\na^\perp\vp = \pa_s \vp - \scal{\pa_s \vp, \tau}\tau,
$$
that is, the normal projection of the arclength derivative of $\vp$ (this operator, restricted to normal vector fields along $\gamma$, coincides with the canonical connection on the normal bundle of $\gamma$ in $\R^n$, which is compatible with the metric). 

We will see in the next section that the ``elastic flow'' associated to the functional $\sE$, of an initial smooth regular curve $\ga_0:\SSS^1\to\R^n$, is given by a smooth solution $\ga:[0,T)\times \SSS^1 \to \R^n$ of the PDE problem
\begin{equation}\label{eq:DefFlusso}
\begin{cases}
\pa_t \ga = -\na^\perp\na^\perp k -|k|^2 k/2 + k\\
\ga(0,\cdot)=\ga_0
\end{cases}
\end{equation}
where $k=k(t,\theta)$ is the curvature vector of the curve $\ga(t,\theta)$ at time $t$.
It is well known (\cite{PoldenThesis} in codimension one and~\cite{DzKuSc02} in any codimension, see also~\cite{Ma02}) that for every initial smooth regular curve $\ga_0$, the elastic flow exists smooth uniquely for every positive time (that is,  $T=+\infty$) and it ``sub--converges'' to a smooth critical point $\ga_\infty:\SSS^1\to\R^n$ of the functional $\sE$. More precisely, we can state the following sub--convergence result.

\begin{prop}[\cite{PoldenThesis, DzKuSc02}]\label{thm:SubConvergence}
Let $\ga_0:\SSS^1\to \R^n$ be a smooth regular curve. Then there exists a unique smooth solution $\ga:[0,+\infty)\times \SSS^1\to \R^n$ of problem~\eqref{eq:DefFlusso}. Moreover, there exist a smooth critical point $\ga_\infty:\SSS^1\to\R^n$ of $\sE$, a sequence of times $t_j\nearrow+\infty$ and a sequence of points $p_j\in\R^n$ such that
\[
\ga(t_j,\cdot)-p_j \xrightarrow[j\to+\infty]{} \ga_\infty,
\]
in $C^m(\SSS^1,\R^n)$ for any $m\in\N$, up to reparametrization.
\end{prop}

Our aim is to show that actually all the flow converges to a critical point $\ga_\infty$, as $t\to+\infty$.

\begin{thm}\label{thm:FullConvergence}
Let  $\ga:[0,+\infty)\times \SSS^1 \to \R^n$ be a smooth solution of the elastic flow, then 
there exists a smooth critical point $\ga_\infty$ of $\sE$ such that
\[
\ga(t,\cdot) \xrightarrow[t\to+\infty]{}\ga_\infty
\]
in $C^m(\SSS^1,\R^n)$ for any $m\in \N$, up to reparametrization. In particular, there exists a compact set $K\con\R^n$ such that $\ga(t,\SSS^1)\con K$ for any time $t\ge0$.
\end{thm}

\begin{remark}
We underline that, even if Theorem~\ref{thm:FullConvergence} implies that the solution of the elastic flow in $\R^n$ stays in a compact region, it does not tell anything about its shape. We believe it is a nice open question to quantify the size of such compact set, depending on the given initial datum $\ga_0$. We also mention that a related problem proposed by G. Huisken is to determine whether the flow starting from a curve in the upper halfplane of $\R^2$, at some time is instead completely contained in the lower halfplane.
\end{remark}

\begin{remark}
We observe that the conclusion of Theorem~\ref{thm:FullConvergence} can be extended to the flow by the gradient of the ``modified'' functional $\sE_\la(\ga)\int_{\SSS^1} \la + \frac12|k|^2\,ds$, for every $\lambda>0$. Moreover, we remark that the same result holds also for the elastic flow of curves in the $2$--dimensional hyperbolic space or in the $2$--dimensional sphere $\SSS^2$. More generally, we expect that it is possible to prove that in a complete, {\em homogeneous} Riemannian manifold $(M^n,g)$ (that is, the group of isometries acts transitively on the manifold), the sub--convergence of the elastic flow can be improved to the full convergence. For a proof of these results and further comments we refer to~\cite{Po20Loja}. We remark that the hypothesis of $(M^n,g)$ being an analytic manifold (with analytic metric $g$) of bounded geometry is not sufficient, see~\cite[Appendix B]{Po20Loja}. 
\end{remark}

\section{The elastic functional}\label{sec:Rn}

We first notice that the elastic functional $\sE$ can be defined on every regular closed curve in $H^4(\SSS^1,\R^n)$, since, by Sobolev embedding theorem, such a curve belongs to $C^3(\SSS^1,\R^n)$, hence its unit tangent and curvature vector fields are well defined and continuous.

Assume that $\gamma:\SSS^1\to\R^n$ is a smooth regular closed curve in $\R^n$ and $\vp\in H^4(\SSS^1,\R^n)$. If $|\varepsilon|$ is small enough, then $\gamma_\varepsilon=\gamma+\varepsilon X\in H^4(\SSS^1,\R^n)$ is still a regular curve, being $\gamma_\varepsilon\in C^3(\SSS^1,\R^n)$ and $C^3$--converging to $\gamma$ as $\varepsilon\to 0$, again by Sobolev embedding theorem. Then, denoting with $\tau_\ep$ and $k_\ep$ its unit tangent and curvature vector fields, respectively and letting $ds_\varepsilon$ to be the arclength measure associated to $\gamma_\varepsilon$, we want to compute the first and second derivatives in $\varepsilon$ of the function 
$$
\varepsilon\mapsto\sE(\gamma_\varepsilon)=\sE(\gamma+\varepsilon X )= \int_{\SSS^1} (1+|k_\varepsilon|^2/2 )\,ds_\varepsilon,
$$
in order to get the first and second variations of $\sE$ at $\gamma$, with the field $ X $ as infinitesimal generator of the ``deformation'' of $\gamma$. 

We will denote with $\partial_\varepsilon$ the partial derivative in $\varepsilon$, which clearly commutes with $\partial_\theta$ but not with $\partial_s$ or $\nabla^\perp$ (see below).

In the next computations, we will need the following straightforward integration by parts formula,
\begin{equation}\label{normalIP}
\int_{\SSS^1}\left\langle \na^\perp\vp ,  Y  \right\rangle\,ds=
-\int_{\SSS^1}\left\langle\vp , \na^\perp  Y  \right\rangle\,ds,
\end{equation}
holding for every couple of normal vector fields $\vp, Y  \in H^1(\SSS^1,\R^n)$ along $\gamma$, coming from the standard formula
$$
\int_{\SSS^1}\left\langle \partial_s\vp ,  Y  \right\rangle\,ds=
-\int_{\SSS^1}\left\langle\vp , \partial_s Y  \right\rangle\,ds,
$$
for every couple of general vector fields $\vp, Y  \in H^1(\SSS^1,\R^n)$.\\
Moreover, if $\vp:\SSS^1\to\R^n$ is a vector field along $\gamma$, we will denote with $\vp^\top$ and $\vp^\perp$, respectively the projection on the tangent or normal space of $\ga$, that is, 
$$
\vp^\top (\theta) = \scal{\vp(\theta),\tau(\theta)}\tau(\theta)\qquad\text{ and }\qquad \vp^\perp(\theta) = \vp(\theta) - \vp^\top(\theta).
$$

It is easy to compute the variation of the arclength measure $ds_\varepsilon$ associated to $\gamma_\varepsilon$,
\[
\begin{split}
\partial_\varepsilon ds_\varepsilon&=\partial_\varepsilon |\partial_\theta\gamma_\varepsilon|\,d\theta=\frac{\langle\partial_\varepsilon\partial_\theta\gamma_\varepsilon,\partial_\theta\gamma_\varepsilon\rangle}{|\partial_\theta\gamma_\varepsilon|}\,d\theta=
\frac{\langle\partial_\theta\partial_\varepsilon\gamma_\varepsilon,\tau_\varepsilon\rangle}{|\partial_\theta\gamma_\varepsilon|}\,ds_\varepsilon=\langle\partial_s X ,\tau_\varepsilon\rangle\,ds_\varepsilon\\
&=\bigl[\partial_s\langle X ,\tau_\varepsilon\rangle-
\langle X ,k_\varepsilon\rangle\bigr]\,ds_\varepsilon
\end{split}
\]
as $\partial_\varepsilon\gamma_\varepsilon= X $. In order to proceed, we need the following ``commutation'' formula:
$$
\partial_\varepsilon \partial_s f=\partial_\varepsilon \frac{\partial_\theta f}{|\gamma_\varepsilon'|}
=\frac{1}{|\gamma_\varepsilon'|}\partial_\varepsilon \partial_\theta f
-\Bigl\langle\frac{\partial_\theta\gamma_\varepsilon}{|\partial_\theta\gamma_\varepsilon|^3},\partial_\varepsilon \partial_\theta\gamma_\varepsilon\Bigr\rangle\,\partial_\theta f
=\partial_s\partial_\varepsilon f-\bigl\langle\tau_\varepsilon,\partial_s X \bigr\rangle\,\partial_s f,
$$
for every function $f:\SSS^1\to\R$. Hence, we can write
\begin{equation}\label{commut1}
\partial_\varepsilon \partial_s=\partial_s\partial_\varepsilon -\langle\tau_\varepsilon,\partial_s X \rangle\,\partial_s
=\partial_s\partial_\varepsilon -\partial_s\langle\tau_\varepsilon, X \rangle\,\partial_s+
\langle k_\varepsilon, X \rangle\,\partial_s.
\end{equation}
Then, we compute
\begin{align}
\partial_\varepsilon \tau_\varepsilon=&\,\partial_\varepsilon \partial_s\gamma_\varepsilon\nonumber\\
=&\,\partial_s\partial_\varepsilon\gamma_\varepsilon -\langle\tau_\varepsilon,\partial_s X \rangle\,\partial_s\gamma_\varepsilon\nonumber\\
=&\,\partial_s X  -\langle\tau_\varepsilon,\partial_s X \rangle\,\tau_\varepsilon\nonumber\\
=&\,[\partial_s X ]^\perp\nonumber\\
=&\,[\partial_s(\langle\tau_\varepsilon, X \rangle\tau_\varepsilon+ X^\perp)]^\perp\nonumber\\
=&\,\nabla^\perp X^\perp+\langle\tau_\varepsilon, X \rangle k_\varepsilon\label{taueq}
\end{align}
and
\begin{align}
\partial_\varepsilon k_\varepsilon=&\,\partial_\varepsilon \partial_s\tau_\varepsilon\nonumber\\
=&\,\partial_s\partial_\varepsilon\tau_\varepsilon -\partial_s\langle\tau_\varepsilon, X \rangle\,\partial_s\tau_\varepsilon+\langle k_\varepsilon, X \rangle\,\partial_s\tau_\varepsilon\nonumber\\
=&\,\partial_s[\nabla^\perp X^\perp+\langle\tau_\varepsilon, X \rangle k_\varepsilon]
-\partial_s\langle\tau_\varepsilon, X \rangle\,k_\varepsilon+\langle k_\varepsilon, X \rangle\,k_\varepsilon\nonumber\\
=&\,\nabla^\perp\nabla^\perp X^\perp+\langle\partial_s\nabla^\perp X^\perp,\tau_\varepsilon\rangle\tau_\varepsilon
+\langle\tau_\varepsilon, X \rangle \partial_s k_\varepsilon+\langle k_\varepsilon, X \rangle\,k_\varepsilon\nonumber\\
=&\,\nabla^\perp\nabla^\perp X^\perp-\langle\nabla^\perp X^\perp,k_\varepsilon\rangle\tau_\varepsilon
+\langle\tau_\varepsilon, X \rangle \partial_s k_\varepsilon+\langle k_\varepsilon, X \rangle\,k_\varepsilon,\label{keq}
\end{align}
where we canceled the scalar products between orthogonal vectors. We then also get
\begin{align*}
\partial_\varepsilon|k_\varepsilon|^2
=&\,2\langle k_\varepsilon,\nabla^\perp\nabla^\perp X^\perp\rangle
+2\langle\tau_\varepsilon, X \rangle\,\langle k_\varepsilon,\partial_s k_\varepsilon\rangle+2\langle k_\varepsilon, X \rangle\,|k_\varepsilon|^2,
\end{align*}
which implies the first variation formula
\begin{align}
\de \sE_{\gamma_\varepsilon}( X )\coloneqq &\,\frac{d\,}{d\varepsilon}\sE(\gamma_\varepsilon)\nonumber\\
=&\,\int_{\SSS^1}\Bigl[\langle k_\varepsilon,\nabla^\perp\nabla^\perp X^\perp\rangle
+\langle\tau_\varepsilon, X \rangle\,\langle k_\varepsilon,\partial_s k_\varepsilon\rangle+\langle k_\varepsilon, X \rangle\,|k_\varepsilon|^2\nonumber\\
&\,\phantom{\int_{\SSS^1}\Bigl[}+(1+|k_\varepsilon|^2/2)\bigl[\partial_s\langle\tau_\varepsilon, X \rangle-
\langle k_\varepsilon, X \rangle\bigr]\,\Bigr]\,ds_\varepsilon\nonumber\\
=&\,\int_{\SSS^1}\Bigl[\langle \nabla^\perp\nabla^\perp k_\varepsilon, X^\perp\rangle
+\langle\tau_\varepsilon, X \rangle\,\langle k_\varepsilon,\partial_s k_\varepsilon\rangle+\langle k_\varepsilon, X \rangle\,|k_\varepsilon|^2/2\nonumber\\
&\,\phantom{\int_{\SSS^1}\Bigl[}+\partial_s\langle\tau_\varepsilon, X \rangle|k_\varepsilon|^2/2-\langle k_\varepsilon, X \rangle\,\Bigr]\,ds_\varepsilon\nonumber\\
=&\,\int_{\SSS^1}\Bigl[\langle\nabla^\perp\nabla^\perp k_\varepsilon, X \rangle
+\langle k_\varepsilon, X \rangle\,|k_\varepsilon|^2/2-\langle k_\varepsilon, X \rangle\,\Bigr]\,ds_\varepsilon\nonumber\\
=&\,\int_{\SSS^1}\bigl\langle \na^\perp\na^\perp k_\varepsilon + |k_\varepsilon|^2k_\varepsilon/2 - k_\varepsilon,  X\bigr\rangle\,ds_\varepsilon,\label{eq:RefFirstVar}
\end{align}
where we integrated by parts in the second and third step.

In particular, for any smooth regular curve $\gamma:\SSS^1\to\R^n$, the $L^2(ds)$--gradient of the functional $\sE$, giving rise to the definition of the elastic flow ~\eqref{eq:DefFlusso}, is given by
$$
\na^\perp\na^\perp k + |k|^2k/2 - k
$$
simply by evaluating at $\ep=0$. We notice that the first variation of $\sE$ at $\ga$ only depends on the normal part $ X^\perp$ of the vector field $ X $ along $\ga$, being $\na^\perp\na^\perp k + |k|^2k/2 - k$ a normal vector field along $\gamma$. This well known fact is due to the ``geometric nature'' of the functional $\sE$, in particular to its invariance by reparametrization of the curves.

\begin{remark}
The above computation is also justified if $\gamma$ is a regular curve in $H^4(\SSS^1,\R^n)$ and we are considering the first variation $\de\sE_\gamma$ as an element of $H^4(\SSS^1,\R^n)^\star$, defined by $\de\sE_\gamma(X)= \tfrac{d}{d\varepsilon}\sE(\ga+\varepsilon X)\bigr\vert_{\varepsilon=0}$. Indeed, $\de\sE_\gamma \in L^2(\SSS^1,\R^n)^\star$ and it is represented by the normal vector field $|\ga'|\bigl(\na^\perp\na^\perp k + |k|^2k/2 - k\bigr)$ along $\gamma$, with respect to the $L^2(d\theta)$--scalar product (and by $\na^\perp\na^\perp k + |k|^2k/2 - k$ with respect to the $L^2(ds)$--scalar product).

A critical point of $\sE$ is a regular curve $\ga:\SSS^1\to\R^n$ of class $H^4$ such that $\delta\sE_\ga=0$, that is, $\na^\perp\na^\perp k + |k|^2k/2 - k = 0$. Standard regularity arguments imply that such a critical point is actually of class $C^\infty$ (see for example the proof of~\cite[Proposition 4.1]{DaPl17}). In particular, an elastic flow~\eqref{eq:DefFlusso} starting from a critical point simply does not move.
\end{remark}

Before dealing with the second variation of $\sE$, we work out another commutation formula:
\begin{align*}
\partial_\varepsilon \nabla^\perp Y 
=&\,\partial_\varepsilon[\partial_s Y -\langle \partial_s Y ,\tau_\varepsilon\rangle\tau_\varepsilon]\\
=&\,\partial_\varepsilon\partial_s Y -\langle \partial_\varepsilon\partial_s Y ,\tau_\varepsilon\rangle\tau_\varepsilon-\langle\partial_s Y ,\partial_\varepsilon\tau_\varepsilon\rangle\tau_\varepsilon-\langle \partial_s Y ,\tau_\varepsilon\rangle \partial_\varepsilon\tau_\varepsilon\\
=&\,[\partial_\varepsilon\partial_s Y ]^\perp-\langle\partial_s Y ,\partial_\varepsilon\tau_\varepsilon\rangle\tau_\varepsilon-\langle \partial_s Y ,\tau_\varepsilon\rangle \partial_\varepsilon\tau_\varepsilon\\
=&\,[\partial_s\partial_\varepsilon Y  -\langle\tau_\varepsilon,\partial_s X \rangle\,\partial_s Y ]^\perp-\langle\partial_s Y ,\partial_\varepsilon\tau_\varepsilon\rangle\tau_\varepsilon-\langle \partial_s Y ,\tau_\varepsilon\rangle \partial_\varepsilon\tau_\varepsilon\\
=&\,\nabla^\perp\partial_\varepsilon Y  -\langle\tau_\varepsilon,\partial_s X \rangle\,\nabla^\perp Y -\langle\partial_s Y ,\nabla^\perp X^\perp+\langle\tau_\varepsilon, X \rangle k_\varepsilon\rangle\tau_\varepsilon\nonumber\\
&\,-\langle \partial_s Y ,\tau_\varepsilon\rangle(\nabla^\perp X^\perp+\langle\tau_\varepsilon, X \rangle k_\varepsilon)\\
=&\,\nabla^\perp\partial_\varepsilon Y 
-\langle\tau_\varepsilon,\partial_s X \rangle\,\nabla^\perp Y 
-\langle\nabla^\perp Y ,\nabla^\perp X^\perp\rangle\tau_\varepsilon
-\langle\tau_\varepsilon, X \rangle\langle\nabla^\perp Y ,k_\varepsilon\rangle\tau_\varepsilon\\
&\,-\langle \partial_s Y ,\tau_\varepsilon\rangle\nabla^\perp X^\perp
-\langle \partial_s Y ,\tau_\varepsilon\rangle\langle\tau_\varepsilon, X \rangle k_\varepsilon
\end{align*}
where we used commutation formula~\eqref{commut1} and~\eqref{taueq}.

In particular, if $Y= Y(\ep)$ is a normal vector field along $\ga_\ep$ for any $\ep$, carrying in the last line the $\partial_s$ derivative out of the scalar products, we get
\begin{align*}
\partial_\varepsilon \nabla^\perp Y 
=&\,\nabla^\perp\partial_\varepsilon Y 
-\langle\tau_\varepsilon,\partial_s X \rangle\,\nabla^\perp Y 
-\langle\nabla^\perp Y ,\nabla^\perp X^\perp\rangle\tau_\varepsilon
-\langle\tau_\varepsilon, X \rangle\langle\nabla^\perp Y ,k_\varepsilon\rangle\tau_\varepsilon\\
&\,+\langle  Y ,k_\varepsilon\rangle\nabla^\perp X^\perp
+\langle  Y ,k_\varepsilon\rangle\langle\tau_\varepsilon, X \rangle k_\varepsilon
\end{align*}
and if also $ X $ is normal along $\ga$, working analogously we conclude
\begin{equation}\label{commut2bisbis}
\partial_\varepsilon \nabla^\perp Y 
=\nabla^\perp\partial_\varepsilon Y 
+\langle X , k_\varepsilon\rangle\,\nabla^\perp Y 
-\langle\nabla^\perp Y ,\nabla^\perp X \rangle\tau_\varepsilon
+\langle  Y ,k_\varepsilon\rangle\nabla^\perp X
\end{equation}
at $\ep=0$.

By means of the above conclusion, we can write
\begin{equation}\label{secvar}
\de^2 \sE_\gamma(X,X)\coloneqq
  \frac{d^2\,}{d\varepsilon^2}\sE(\gamma + \varepsilon X)\,\Bigl\vert_{\varepsilon=0}
= \frac{d\,}{d\varepsilon}\bigl\langle \na^\perp\na^\perp k_\varepsilon
+ |k_\varepsilon|^2k_\varepsilon/2 - k_\varepsilon,  X  \bigr\rangle_{L^2(ds_\varepsilon)}\,\Bigl\vert_{\varepsilon=0},
\end{equation}
that is,
\begin{align*}
\de^2 \sE_\gamma(X,X)=&\,\int_{\SSS^1}\bigl\langle\partial_\varepsilon\bigl( \na^\perp\na^\perp k_\varepsilon
+ |k_\varepsilon|^2k_\varepsilon/2 - k_\varepsilon\bigr)\,\bigr\vert_{\varepsilon=0},  X\bigr\rangle\,ds\\
&\,+\int_{\SSS^1}\bigl\langle \na^\perp\na^\perp k + |k|^2k/2 - k,  X\bigr\rangle\,\bigl[\partial_s\langle X ,\tau\rangle-\langle X ,k\rangle\bigr]\,ds.
\end{align*}
Since this is the case we are interested in, we assume that $\gamma$ is a critical point of $\sE$ (that is, $\delta\sE_\ga=0$) and $X$ is a normal vector field along $\gamma$, hence
\begin{equation}\label{2der}
\de^2 \sE_\gamma(X,X)=\int_{\SSS^1}\bigl\langle\partial_\varepsilon\bigl( \na^\perp\na^\perp k_\varepsilon
+ |k_\varepsilon|^2k_\varepsilon/2 - k_\varepsilon\bigr)\,\bigr\vert_{\varepsilon=0},  X\bigr\rangle\,ds
\end{equation}
being the second line above equal to zero, as $ \na^\perp\na^\perp k + |k|^2k/2 - k=0$.\\
Assuming that $ X $ is a normal vector field along $\gamma$, by means of equations~\eqref{commut2bisbis} and~\eqref{keq}, we have
\begin{align*}
\partial_\varepsilon\na^\perp\na^\perp k_\varepsilon\,\bigr\vert_{\varepsilon=0}
=&\,\nabla^\perp\partial_\varepsilon\na^\perp k_\varepsilon\,\bigr\vert_{\varepsilon=0}
+\langle k, X \rangle\,\nabla^\perp\na^\perp k
-\langle\nabla^\perp\na^\perp k,\nabla^\perp X \rangle\tau
+\langle \na^\perp k,k\rangle\nabla^\perp X \\
=&\,\nabla^\perp\bigl[
\nabla^\perp\partial_\varepsilon k_\varepsilon\,\bigr\vert_{\varepsilon=0}
+\langle X , k\rangle\,\nabla^\perp k
-\langle\nabla^\perp k,\nabla^\perp X \rangle\tau
+\langle  k,k\rangle\nabla^\perp X ]\\
&\,+\langle k, X \rangle\,\nabla^\perp\na^\perp k
-\langle\nabla^\perp\na^\perp k,\nabla^\perp X \rangle\tau
+\langle \na^\perp k,k\rangle\nabla^\perp X \\
=&\,\nabla^\perp\nabla^\perp\bigl[
\nabla^\perp\nabla^\perp X -\langle\nabla^\perp X ,k\rangle\tau
+\langle k, X \rangle\,k\bigr]\\
&\,+\nabla^\perp\bigl[\langle X , k\rangle\,\nabla^\perp k
-\langle\nabla^\perp k,\nabla^\perp X \rangle\tau
+|k|^2\nabla^\perp X ]\\
&\,+\langle k, X \rangle\,\nabla^\perp\na^\perp k
-\langle\nabla^\perp\na^\perp k,\nabla^\perp X \rangle\tau
+\langle \na^\perp k,k\rangle\nabla^\perp X .
\end{align*}
Hence, dropping the scalar products which are zero by orthogonality, we get
$$
\int_{\SSS^1}\bigl\langle\partial_\varepsilon\bigl( \na^\perp\na^\perp k_\varepsilon\,\bigr\vert_{\varepsilon=0},  X \bigr\rangle\,ds
=\int_{\SSS^1}\bigl\langle \na^\perp\na^\perp \nabla^\perp\nabla^\perp X +\Lambda(X),X\bigr\rangle\,ds
$$
where $\Lambda(X)\in L^2(ds)$ is a normal vector field along $\gamma$, depending only on $k,X$ and their ``normal derivatives'' $\na^\perp$ up to the third order, moreover the dependence on $X$ is linear.\\
The computation of the remaining term in equation~\eqref{2der},
$$
\partial_\varepsilon\bigl(|k_\varepsilon|^2k_\varepsilon/2 - k_\varepsilon\bigr)\,\bigr\vert_{\varepsilon=0}
$$
is easier and follows immediately by equation~\eqref{keq}, giving rise to another term similar to $\Lambda(X)$, linear in $X$ and containing only ``normal derivatives'' $\na^\perp$ of $k$ and $X$ up to the second order.\\
Hence, we conclude
\begin{equation*}
\de^2 \sE_\gamma(X,X)=\int_{\SSS^1}\bigl\langle\na^\perp\nabla^\perp \na^\perp\nabla^\perp\vp +\Omega( X), X\bigr\rangle\,ds
\end{equation*}
where $\Om(X)\in L^2(ds)$ is a normal vector field along $\gamma$, linear in $X$ and depending only on $k,X$ and their ``normal derivatives'' $\na^\perp$ up to the order three.\\
By polarization, we get the symmetric bilinear form on the space of the normal vector fields along $\gamma$ in $H^4(\SSS^1,\R^n)$, giving the second variation of the functional $\sE$ at $\gamma$:
$$
\delta^2\sE_\gamma(X,Y)=\int_{\SSS^1}\bigl\langle \na^\perp\nabla^\perp \na^\perp\nabla^\perp\vp+\Omega(X), Y\bigr\rangle\,ds=\bigl\langle\cL(\vp), Y\bigr\rangle_{L^2(\SSS^1,\R^n)},
$$
where we set
$$
\cL(\vp)\coloneqq |\ga'|\left((\na^\perp)^4 \vp  + \Om(\vp)\right).
$$

\begin{remark}\label{secvar2} 
We observe that $\cL$ and $\Omega$ are linear and continuous maps defined on the space of normal vector fields along $\gamma$ in $H^4(\SSS^1,\R^n)$ and taking values in the normal vector fields along $\gamma$ in $L^2(\SSS^1,\R^n)$, moreover $\Om$ is a compact operator, by Sobolev embeddings. Therefore, for any normal vector field $X$ in $H^4(\SSS^1,\R^n)$, we have that $\de^2\sE_\ga(X,\cdot)$ can be seen as an element of the dual of the space of the normal vector fields along $\gamma$ in $L^2(\SSS^1,\R^n)$.
\end{remark}

\begin{remark}
We refer to~\cite[Section 3.1]{Po20Loja} for the explicit full computation of the first and second variations of $\sE$ in the general case of curves on manifolds, even without assuming that $\gamma$ is a critical point of $\sE$ and that $X,Y$ are normal vector fields. For our purpose here, the previous computations are sufficient.
\end{remark}

\begin{defn}
Let $\ga:\SSS^1\to\R^n$ be a regular smooth closed curve in $\R^n$ and $\tau$ its unit tangent vector field. For $m \in \N$ we define the Sobolev spaces of normal vector fields along $\gamma$ as
$$
H^{m,\perp}_\ga  = \left\{  \vp\in W^{m,2}(\SSS^1,\R^n)\,:\,\text {$\scal{\tau(\theta),\vp(\theta)}=0$ for almost every $\theta\in\SSS^1$}\right\},
$$
where as usual $W^{0,2}(\SSS^1,\R^n)=L^2(\SSS^1,\R^n)$. Moreover, we denote with $L^{2,\perp}_\ga =H^{0,\perp}_\ga $ the normal vector fields along $\gamma$, belonging to $L^2(\SSS^1,\R^n)$.\\
We underline that, unless otherwise stated, the spaces $L^p$ are endowed with the Lebesgue measure $d\theta$. In case it is convenient to consider another measure, like the arclength measure $ds$ of a curve, we will specify $L^p(ds)$. Observe that if $\ga:\SSS^1\to\R^n$ is smooth and regular, then clearly $L^p(d\theta)=L^p(ds)$, for any $p\in[1,+\infty)$.
\end{defn}

We conclude this section by showing that the second variation operator $\de^2\sE_\ga$ is Fredholm of index zero. We recall that by Remark~\ref{secvar2}, using this definition, we can consider $\delta^2\sE_\gamma : H^{4,\perp}_\ga \to (L^{2,\perp}_\ga)^\star$. 

\begin{lemma}\label{lem:FredholmRn}
Let $\ga:\SSS^1\to\R^n$ be a smooth regular curve. The operator $(\na^\perp)^4:H^{4,\perp}_\ga \to L^{2,\perp}_\ga $ is Fredholm of index zero, and then same holds for the operators $\cL:H^{4,\perp}_\ga \to L^{2,\perp}_\ga $ and $\delta^2\sE_\gamma : H^{4,\perp}_\ga \to (L^{2,\perp}_\ga)^\star$. 
\end{lemma}
\begin{proof}
Since $\de^2\sE_\ga(X,Y)=\scal{\cL(X),Y}_{L^2(\SSS^1,\R^n)}$ and $\ga$ is regular, we have that $\delta^2\sE_\gamma : H^{4,\perp}_\ga \to (L^{2,\perp}_\ga)^\star$ is Fredholm of index zero if and only if $(\na^\perp)^4  + \Om:H^{4,\perp}_\ga \to L^{2,\perp}_\ga$ is such. The operator $\Om:H^{4,\perp}_\ga \to L^{2,\perp}_\ga$ is compact, thus $(\na^\perp)^4  + \Om:H^{4,\perp}_\ga \to L^{2,\perp}_\ga $ is Fredholm of index zero if and only if the same holds for $(\na^\perp)^4:H^{4,\perp}_\ga \to L^{2,\perp}_\ga $ (see~\cite[Section 19.1, Corollary 19.1.8]{HormanderIII}), and this happens if the operator $\id + (\na^\perp)^4:H^{4,\perp}_\ga \to L^{2,\perp}_\ga $ is invertible, where $\id$ is the identity/inclusion map.

Indeed if $X\in H^{4,\perp}_\ga $ is in the kernel of $\id + (\na^\perp)^4$, then there must hold
$$
0=\int_{\SSS^1}\bigl\langle \na^\perp\nabla^\perp \na^\perp\nabla^\perp\vp+X,X\bigr\rangle\,ds=
\int_{\SSS^1}\vert\na^\perp\nabla^\perp X\vert^2+\vert X\vert^2\,ds,
$$
by means of~\eqref{normalIP}, which implies that $X=0$, and then $\id + (\na^\perp)^4$ is injective.

It remains to prove that $\id + (\na^\perp)^4:H^{4,\perp}_\ga \to L^{2,\perp}_\ga $ is surjective.
Let $Y\in L^{2,\perp}_\ga $ and consider the continuous functional $\sF:H^{2,\perp}_\ga  \to \R$ defined by
\[
\sF(X) = \int_{\SSS^1} \frac12 \bigl|(\nabla^\perp)^2X\bigr|^2 + \frac12 |X|^2 - \scal{X, Y}\,ds.
\]
An explicit computation shows that
\begin{equation}\label{eq:DoubleNormalDerivative}
(\nabla^\perp)^2X = \partial^2_sX + \left( 2\langle \partial_sX,k\rangle + \langle X, \partial_s k\rangle \right) \tau  - \langle \partial_sX,\tau\rangle k,
\end{equation}
hence,
\[
\int_{\SSS^1} |\partial^2_sX|^2\,ds \le C(\gamma) \int_{\SSS^1}  |(\nabla^\perp)^2X|^2 + |\partial_sX|^2+ |X|^2\,ds.
\]
Then, since 
$$
\int_{\SSS^1} |\partial_sX|^2\,ds = - \int_{\SSS^1}\scal{X,\partial_s^2 X}\,ds \le \varepsilon\int_{\SSS^1} |\partial_s^2X|^2 + C(\varepsilon) \int_{\SSS^1} |X|^2,
$$
for a suitable constant $C(\varepsilon)$, we conclude 
\[
\int_{\SSS^1} |X|^2 +  |\partial_sX|^2 +  |\partial^2_sX|^2\,ds \le C(\gamma) \int_{\SSS^1}  \frac12 |(\nabla^\perp)^2X|^2 + \frac12 |X|^2\,ds,
\]
which implies that the functional $\sF$ is coercive. Therefore, by the direct methods of calculus of variations, it follows that there exists a minimizer $Z$ of $\sF$ in $H^{2,\perp}_\ga $. In particular $Z$ satisfies
\begin{equation}\label{eq:EquationForSurjectivity}
\int_{\SSS^1} \bigl\langle(\nabla^\perp)^2 Z, (\nabla^\perp)^2 X\bigr\rangle + \scal{Z,X}\,ds = \int_{\SSS^1} \scal{Y, X}\,ds,
\end{equation}
for every $X\in H^{2,\perp}_\ga $. If we show that $Z \in H^{4,\perp}_\ga $, then $Z + (\nabla^\perp)^4Z = Y$, and surjectivity is proved. This follows by standard techniques, simply noticing that once writing the integrand of the functional $\sF$ in terms of $\partial_s^2X$, $\partial_sX$ and $X$, by means of equation~\eqref{eq:DoubleNormalDerivative}, its dependence on the highest order term $\partial_s^2X$ is quadratic and the ``coefficients'' are given by the geometric quantities of $\gamma$, which is smooth.
\end{proof}

\section{An abstract \L ojasiewicz--Simon gradient inequality}

In this section we present a result from~\cite{Po20Loja} collecting some conditions under which a \L ojasiewicz--Simon gradient inequality holds (see~\cite{Loja63, Loja84,Si83} for a given energy functional. This result is stated in a purely functional analytic setting for an abstract energy functional, and it can be possibly applied to different evolution equations.

Following~\cite{Ch03}, we assume that $V$ is a Banach space, $U\con V$ is open and $\sE:U\to\R$ is a map of class $C^2$. We denote with $\delta\sE:U\to V^\star$ the differential and with $\sH:U\to L(V,V^\star)$ the second differential (or Hessian) of $\sE$, respectively. We assume that $0\in U$ and we set $V_0=\ker \sH(0)\con V$.\\
We recall that a closed subspace $S\con V$ is said to be {\em complemented} if there exists a continuous {\em projection} $P:V\to V$ such that $\Imm P = S$ (a continuous projection is a linear and continuous map $P:V\to V$ such that $P\circ P = P$). In such a case, we denote by $P^\star:V^\star\to V^\star$ the {\em adjoint projection}.

\begin{prop}[{\cite[Corollary~3.11]{Ch03}}]\label{cor:Chill}
Under the above notation, assume that $\sE:U\to\R$ is analytic and $0\in U$ is a critical point of $\sE$, that is, $\delta\sE(0)=0$. Assume that $V_0$ is finite dimensional (therefore, it is complemented and has a projection map $P:V_0\to V_0$) and there exists a Banach space $W\hookrightarrow V^\star$ (that is, we identify $W$ with a subset of $V^\star$) such that
\begin{enumerate}[label=(\roman*)]
\item $\mathrm{Imm}\,\delta\sE\subseteq W$ and the map $\delta\sE:U\to W$ is analytic (with the norm of $W$), \label{IT.1}		
\item $P^\star(W)\con W$, \label{IT.2}
\item $\sH(0)(V) = \ker P^\star \cap W$. \label{IT.3}
\end{enumerate}
Then, there exist constants $C,\ro>0$ and $\al\in(0,1/2]$ such that
\begin{equation*}
|\sE(u)-\sE(0)|^{1-\al}\le C \| \delta\sE(u) \|_{W},
\end{equation*}
for any $u \in B_\ro(0)\subseteq U$.
\end{prop}

This proposition is a special case of  Corollary~3.11 in~\cite{Ch03}, choosing $X=V$ and $Y=W$ therein. We can then prove the following consequence.

\begin{cor}[{\cite[Corollary~2.6]{Po20Loja}}]\label{cor:LojaFunctionalAnalytic}
Let $\sE:U\con V \to \R$ be an analytic map, where $V$ is a Banach space and $0\in U$ is a critical point of $\sE$. Suppose that we have a Banach space $W=Z^\star\hookrightarrow V^\star$, where $V\hookrightarrow Z$, for some Banach space $Z$, that $\mathrm{Imm}\,\delta\sE\subseteq W$ and the map $\delta\sE:U\to W$ is analytic (with the norm of $W$).  Assume also that $\sH(0)\in L(V,W)$ and it is Fredholm of index zero.\\
Then the hypotheses of Proposition~\ref{cor:Chill} are satisfied, and then there exist constants $C$, $\ro>0$ and $\al\in(0,1/2]$ such that
\begin{equation*}
|\sE(u)-\sE(0)|^{1-\al}\le C \| \delta\sE(u) \|_{W},
\end{equation*}
for any $u \in B_\ro(0)\subseteq U$.
\end{cor}
\begin{proof}
Let us denote $\cH\coloneqq \sH(0):V\to W$. By the hypotheses, the subspace $V_0 \coloneqq \ker \cH $ is finite dimensional, thus it is closed and complemented with a projection $P:V\to V$ such that $\Imm P = V_0$, moreover point~\ref{IT.1} of Proposition~\ref{cor:Chill} is satisfied.\\
We can write $V=V_0 \oplus V_1$, where $V_1=\ker P$, then if $P^\star:V^\star\to V^\star$ is the adjoint projection, we see that also $V^\star= V_0^\star \oplus V_1^\star$ and
\[
V_0^\star = \Imm P^\star, \qquad V_1^\star = \ker P^\star.
\]
We let $J_0:Z\to Z^{\star\star}$ to be the canonical isometric injection and we call $J:V\to Z^{\star\star}$ the restriction of $J_0$ to $V$. 
We claim that $\cH:V\to W$ satisfies 
\begin{equation}\label{eq:QuasiSelfAdjoint}
\cH^\star \circ J=\cH .
\end{equation}
where $\cH^\star: W^\star \to V^\star$ is the adjoint of $\cH$.\\
Indeed, since $\cH$ is symmetric (it is a second differential), for any $ v, u \in V$ and $F\coloneqq J( u ) \in J(V)\con Z^{\star\star}$ we find
\[
(\cH^\star \circ J )( u ) [ v] = \cH^\star(F)[ v]=F(\cH v)= J( u )(\cH v) = \cH v[ u ] =\cH( u )[ v].
\]
As a general consequence of the fact that $\cH$ is Fredholm of index zero, we have
\[
{\rm dim} \ker \cH= {\rm dim} \ker \cH^\star,
\]
indeed, index zero means that ${\rm dim}\ker \cH = \dim  {\mathrm{coker}}\, \cH$, where we split $W$ as
\[
W= \Imm \cH\oplus {\mathrm{coker}}\, \cH,
\]
and ${\mathrm{coker}}\, \cH$ is finite dimensional. Therefore, $W^\star = (\Imm \cH)^\star  \oplus ({\mathrm{coker}} \,\cH)^\star$ and since $\ker \cH^\star = (\Imm\cH)^\perp = ({\mathrm{coker}}\, \cH)^\star$, we conclude that 
$$
\dim \ker \cH^\star = \dim ({\mathrm{coker}}\, \cH)^\star = \dim {\mathrm{coker}}\, \cH = \dim \ker \cH.
$$
We claim that
\begin{equation}\label{eq:3}
J(\Imm P) = \ker \cH^\star \cap J(V).
\end{equation}
Indeed, by equality~\eqref{eq:QuasiSelfAdjoint} we see that
\[
\ker \cH = \ker (\cH^\star\circ J) = J^{-1}(\ker\cH^\star)
\]
and applying then $J$ on both sides, we get $J(\Imm P) =\ker\cH^\star\cap J(V)$, that is, formula~\eqref{eq:3}.\\ 
Since $\Imm P = \ker \cH$ and $J$ is injective, we have $\dim\ker\cH={\rm dim} (J(\Imm P)) = {\rm dim} \ker \cH^\star \cap J(V)$. Then, as ${\rm dim} \ker \cH = {\rm dim}\ker \cH^\star$, it follows that $\ker \cH^\star \cap J(V)= \ker \cH^\star$ and
\[
J(\Imm P) = \ker \cH^\star.
\]
Therefore, recalling that $V^{\star\star}\hookrightarrow W^\star$ and that $W\hookrightarrow V^\star$, we get
\begin{equation}
\begin{split}
(\ker\cH^\star)^\perp &= \bigl\{ w \in W\,\,:\,\,\scal{f,w}_{W^\star,W}=0\,\,\,\forall f \in J(\Imm P) \bigr\} \\
&= \bigl\{ w \in W\,\,:\,\,\scal{J(v),w}_{W^\star,W}=0\,\,\,\forall v \in \Imm P \bigr\} \\
&= \bigl\{ w \in W\,\,:\,\,\scal{w,v}_{V^\star,V}=0\,\,\,\forall v \in \Imm P \bigr\} \\
&= (\Imm P)^\perp \cap W.
\end{split}
\end{equation}
Finally, as $\Imm\cH$ is closed,  we have
\[
\begin{split}
\Imm\cH&= (\ker \cH^\star)^\perp\\
& = (\Imm P)^\perp \cap W \\
&= \bigl\{ 	f \in V^\star\,\,\,:\,\,\,\scal{f,Pv}_{V^\star,V}=0\,\,\,\forall v \in V	\bigr\} \cap W \\
&= \ker P^\star \cap W,
\end{split}
\]
then point~\ref{IT.3} of Proposition~\ref{cor:Chill} is verified.\\
We are just left with proving point~\ref{IT.2}, that is, $P^\star(Z^\star)\con Z^\star$. We observe that if we check that $P^\star(Z^\star\cap V_0^\star) \con Z^\star \cap V_0^\star$, then we are done, indeed we would get
\[
P^\star(Z^\star)= P^\star ( Z^\star \cap V_0^\star \oplus Z^\star \cap V_1^\star  ) = P^\star (Z^\star\cap V_0^\star) \con Z^\star \cap V_0^\star \con Z^\star.
\]
If $f_0\in Z^\star \cap V_0^\star$, writing any $ v\in V$ as $ v= v_0\oplus  v_1 \in V_0\oplus V_1$, we get
\[
P^\star(f_0)[ v]=f_0(P v)=f_0( v_0) = f_0( v_0)+f_0( v_1) = f_0( v),
\]
indeed,
$$
f_0( v_1)=(P^\star f_0)( v_1)=f_0(P v_1)=f_0(0)=0.
$$
Hence, we proved that $P^\star f_0=f_0$ for any $f_0\in Z^\star \cap V_0^\star$, thus we got that $P^\star(Z^\star\cap V_0^\star) \con Z^\star \cap V_0^\star$.	
\end{proof}

We mention that a result equivalent to this corollary has been recently proved independently in~\cite{Ru20}.

\section{Convergence of the elastic flow in the Euclidean space}\label{sec:Convergence}

As we said at the beginning of Section~\ref{sec:Rn}, if $\ga:\SSS^1\to\R^n$ is a regular closed curve in $H^4(\SSS^1,\R^n)\hookrightarrow C^3(\SSS^1,\R^n)$, there exists $\ro>0$ such that $\ga+\vp$ is still a regular curve, for any $\vp \in B_\ro(0)\con H^4(\SSS^1,\R^n)$. Moreover, if $\gamma$ is embedded, choosing such $\rho$ small enough, the open set $U=\{x\in\R^n\,:\,d_\ga(x)=d(x,\gamma)<\rho\}$ is a tubular neighborhood of $\gamma$ with the property of {\em unique orthogonal projection}. The ``projection'' map $\pi:U\to\gamma(\SSS^1)$ turns out to be $C^2$ in $U$ and given by $x\mapsto x-\nabla d^2_\ga(x)/2$, moreover the vector $\nabla d^2_\ga(x)$ is orthogonal to $\gamma$ at the point $\pi(x)\in\gamma(\SSS^1)$, see~\cite[Section~4]{mant4} for instance.\\
Then, given a curve $\theta\mapsto\sigma(\theta)=\ga(\theta)+X(\theta)$ with $\vp \in B_\ro(0)\con H^4(\SSS^1,\R^n)$, we define a map $\theta':\SSS^1\to\SSS^1$ as
$$
\theta'=\theta'(\theta)=\gamma^{-1}\bigl[\pi\bigl(\ga(\theta)+X(\theta)\bigr)\bigr],
$$
noticing that it is $C^2$ and invertible if $\ga'(\theta)+\vp'(\theta)$ is never parallel to the unit vector $\nabla d_\ga(\ga(\theta)+X(\theta))$, which is true if we have (possibly) chosen a smaller $\rho$ (hence, $|X|$ and $|X'|$ are small and the claim follows as $\langle \gamma'(\theta),\nabla d_\ga(x)\rangle\to0$, as $x\to\ga(\theta)$).\\
We consider the vector field along $\gamma$,
$$
Y(\theta')=\nabla d^2_\ga(\ga(\theta)+X(\theta))/2
$$
which, for every $\theta'\in\SSS^1$, is orthogonal to $\gamma$ at the point $\pi(\ga(\theta)+X(\theta))=\gamma(\theta')$ by what we said above and the definition of $\theta'=\theta'(\theta)$, hence it is a normal vector field along the curve $\theta'\mapsto\gamma(\theta')$. Thus, we have
\begin{align*}
\gamma(\theta')+Y(\theta')=&\,\pi\bigl(\ga(\theta)+X(\theta)\bigr)+\nabla d^2_\ga(\ga(\theta)+X(\theta))/2\\
=&\,\ga(\theta)+X(\theta)-\nabla d^2_\ga(\ga(\theta)+X(\theta))/2+\nabla d^2_\ga(\ga(\theta)+X(\theta))/2\\
=&\,\ga(\theta)+X(\theta)
\end{align*}
and we conclude that the curve $\sigma=\ga+\vp$ can be described by the (reparametrized) regular curve $\widetilde{\sigma}=\ga+Y$, with $Y$ a normal vector field along $\gamma$ in $H^4(\SSS^1,\R^n)$ as $X$, that is, $Y\in H^{4,\perp}_\ga$. Moreover, it is clear that if $X\to0$ in $H^{4}(\SSS^1,\R^n)$ then also $Y\to0$ in $H^{4,\perp}_\ga$.\\
All this can be done also for a regular curve $\gamma$ which is only {\em immersed} (that is, it can have self--intersections), recalling that locally every immersion is an embedding and repeating the above argument a piece at a time along $\gamma$, getting also in this case a normal field $Y$ describing a curve $\sigma$ which is $H^4$--close enough to $\gamma$, that is $\Vert\sigma-\gamma\Vert_{H^4(\SSS^1,\R^n)}<\rho_\ga$, for some $\rho_\ga>0$, as a ``normal graph'' on $\gamma$, as in the embedded case. 

\medskip

We recall now some further details about the sub--convergence of the elastic flow stated in Proposition~\ref{thm:SubConvergence}. We set $\gamma_t=\gamma(t,\cdot)$ and we let $\ga_\infty$, $t_j$, $p_j$ and $\overline{\gamma}_{t_j}=\overline{\gamma}(t_j,\cdot)$ be the reparametrization of $\gamma_{t_j}$ as in Proposition~\ref{thm:SubConvergence}, then 
\[
\overline{\ga}_{t_j}-p_j \xrightarrow[j\to+\infty]{} \ga_\infty
\]
in $C^m(\SSS^1,\R^n)$ for any $m\in \N$. Moreover, there are positive constants $C_L=C_L(\ga_0)$ and $C(m,\ga_0)$, for any $m\in \N$, such that
\[
\frac{1}{C_L} \le L(\ga_t) \le C_L
\]
and
\begin{equation}\label{eq:StimaParabolica}
\| (\nabla^\perp)^m k(t,\cdot) \|_{L^2(ds)}\le C(m,\ga_0)
\end{equation} 
for every $t\ge0$. These facts follow from the results in~\cite{DzKuSc02,PoldenThesis}, see~\cite[Section 3]{DzKuSc02} in particular.\\
It is then a straightforward computation to see that, if we describe a curve of the flow $\gamma_t=\gamma_\infty+\vp$, which is $H^4$--close enough to $\gamma_\infty$ (precisely, $\vp \in B_\ro(0)\con H^4(\SSS^1,\R^n)$, with $\rho=\rho_{\ga_\infty}$ as above), as a ``normal graph'' on $\gamma_\infty$, that is $\widetilde{\gamma}=\gamma_\infty+Y_t$ with $Y\in H^{4,\perp}_{\ga_\infty}$, we have
\begin{equation}\label{eq:StimaParabolica2}
\|Y_t\|_{H^m}\le C(m,\ga_0,\ga_\infty)\,,
\end{equation} 
for every $m\in\N$.

\begin{defn}
Let $\ga:\SSS^1\to\R^n$ be a regular curve of class $H^4$. We consider $\ro=\ro_\ga>0$ as above and we define the functional
\[
E:B_\ro(0)\con H^{4,\perp}_\ga\to\R \qquad\qquad E(\vp)=\sE(\ga+\vp)\,.
\] 
\end{defn}

By the conclusions of Section~\ref{sec:Rn}, we have 
$$
\de E : B_\ro(0) \con H^{4,\perp}_\ga \to (L^{2,\perp}_\ga)^\star,
$$
given by $X\mapsto\delta E_X=\delta\sE_{\ga+X}$, acting as
$$
\delta E_X(Y)= \left\langle |\ga'+X'|\Bigl(\bigl(\na_{\ga+X}^\perp\bigr)^2k_{\ga+X} + |k_{\ga+X}|^2k_{\ga+X}/2 - k_{\ga+X}\Bigr), Y \right\rangle_{L^2(\SSS^1,\R^n)}
$$
on every $Y\in L^{2,\perp}_\ga$.\\
The second variation $\de^2 E_0$ of $E$ at $0\in H^{4,\perp}_\ga$ clearly coincides with the second variation of $\sE$ at $\ga$, that is,
$$
\de^2 E_0 = \delta^2\sE_\gamma: H^{4,\perp}_\ga \to (L^{2,\perp}_\ga)^\star,
$$
and we have
\[
\delta^2E_0(X,Y) = 
\bigl\langle\cL(\vp), Y\bigr\rangle_{L^2(\SSS^1,\R^n)},
\]
where $\cL(\vp)= |\ga'|\left((\na^\perp)^4 \vp  + \Om(\vp)\right)$.

\begin{prop}\label{prop:Loja}
Let ${\ga_\infty}:\SSS^1\to \R^n$ be a critical point of $\sE$. Then there exist constants $C,\sigma>0$ and $\al \in (0,1/2]$ such that
\begin{equation}\label{eq:LojaFull}
|\sE({\ga_\infty}+Y)-\sE({\ga_\infty})|^{1-\al}\le C \| \delta E_Y \|_{(L^{2,\perp}_{\ga_\infty})^\star}
\end{equation}
for any $Y \in B_\sigma(0)\subseteq H^{4,\perp}_{\ga_\infty}$, where the functional $E$ at the right hand side is relative to the curve $\gamma_\infty$. 
\end{prop}
\begin{proof}
We apply Corollary~\ref{cor:LojaFunctionalAnalytic} to the functional $E:B_{\ro}(0)\con H^{4,\perp}_{\ga_\infty}\to \R$, where $\rho>0$ is  as above, with $V=H^{4,\perp}_{\ga_\infty}$, $W=(L^{2,\perp}_{\ga_\infty})^\star$, and $Z=L^{2,\perp}_{\ga_\infty}$. From the above discussion we have that the first variation $\de E$ (respectively, the second variation $\de^2 E_0$, evaluated at $0$) is defined on $B_{\ro}(0) \con V$ (respectively, on $V$) and it is $W$--valued (the same for $\de^2 E_0$). Moreover, $0\in H^{4,\perp}_{\ga_\infty}$ is a critical point of $E$, by assumption and we have that $\de^2 E_0:V\to W$ is a Fredholm operator of index zero by Lemma~\ref{lem:FredholmRn}, as it coincides with $\de^2\sE_{\ga_\infty}$. Finally, both $E$ and $\de E$ are analytic as maps between $B_{\ro}(0)$ and $\R$, $W$ (with its norm) respectively (this can be proved directly by noticing that $E$ and $\de E$ are compositions and sums of analytic functions -- for a detailed proof of this fact we refer to~\cite[Lemma~3.4]{DaPoSp16}).\\
Therefore, we can apply Corollary~\ref{cor:LojaFunctionalAnalytic} and we conclude that get that there exist constants $C,\sigma>0$ and $\al \in (0,1/2]$ such that
\begin{equation}\label{eq:LojaE}
|E(Y)-E(0)|^{1-\al}\le C \| \de E_Y \|_{(L^{2,\perp}_{\ga_\infty})^\star}
\end{equation}
for any $Y \in B_{\sigma}(0)\con H^{4,\perp}_{\ga_\infty}$ and we are done.
\end{proof}

Now we are ready to prove the full convergence of the flow. 

\begin{proof}[Proof of Theorem~\ref{thm:FullConvergence}]
As before, we set $\gamma_t=\gamma(t,\cdot)$ and we let $\ga_\infty$, $t_j$, $p_j$ and $\overline{\gamma}_{t_j}=\overline{\gamma}(t_j,\cdot)$ be as in Proposition~\ref{thm:SubConvergence}. Moreover, to simplify the notation we denote with $L^2(d\theta)$ the space $L^2(\SSS^1,\R^n)$.\\
We start with noticing that along the flow the elastic functional is monotone nonincreasing as
$$
\frac{d\,}{dt}\sE(\gamma_t)=-\int_{\SSS^1} \bigl\vert (\na^\perp)^2 k - |k|^2 k/2 + k\bigr\vert^2\,ds=-\Vert\partial_t\gamma\Vert_{L^2(ds)}^2
$$
and actually we can assume that it is decreasing in every time interval, otherwise at some time $t_0$ the curve $\gamma_{t_0}$ is a critical point, then the flow stops and the theorem clearly follows. As $i\leq j$ implies $t_i\leq t_j$, we have
$$
\sE(\overline{\ga}_{t_i}-p_i)=\sE(\ga_{t_i})\geq\sE(\ga_{t_j})=\sE(\overline{\ga}_{t_j}-p_j) 
$$
which clearly implies, as $\overline{\ga}_{t_j}-p_j \to \ga_\infty$ in $C^m(\SSS^1,\R^n)$, that $\sE(\ga_{t_i})=\sE(\overline{\ga}_{t_i}-p_i)\geq\sE(\ga_\infty)$, for every $i\in\N$ and $\sE(\ga_t)\searrow\sE(\ga_\infty)$, as $t\to+\infty$.\\
Thus, it is well defined the following positive function
\[
H(t)= \left[\sE(\ga_t) - \sE(\ga_\infty) \right]^\al,
\]
where $\al\in(0,1/2]$ is given by Proposition~\ref{prop:Loja}. The function $H$ is monotone decreasing and converging to zero as $t\to+\infty$ (hence, bounded above by $H(0)= \left[\sE(\ga_0) - \sE(\ga_\infty) \right]^\al$).

Now let $m\ge6$ be a fixed integer. By Proposition~\ref{thm:SubConvergence}, for any $\ep>0$ there exists $j_\ep\in\N$ such that
\[
\| \overline{\ga}_{t_{j_\ep}} - p_{j_\ep} -\ga_\infty \|_{C^m(\SSS^1,\R^n)}\le \ep\qquad\text{ and }\qquad H(t_{j_\ep})\le\ep.
\]
Choosing $\varepsilon>0$ small enough, in order that 
$$
(\overline{\ga}_{t_{j_\ep}} - p_{j_\ep} -\ga_\infty)\in B_{\rho_{\ga_\infty}}(0)\con H^4(\SSS^1,\R^n),
$$
by the argument at the beginning of this section (with $\gamma=\gamma_\infty$), for every $t$ in some interval $[t_{j_\varepsilon},t_{j_\ep} + \de)$ there exists $Y_t \in H^{4,\perp}_{\ga_\infty}$ such that the curve $\widetilde{\ga}_t = \ga_\infty + Y_t$ is the ``normal graph'' reparametrization of $\gamma_t-p_{j_\ep}$, hence 
$$
(\pa_t \widetilde{\ga})^\perp = -(\na_{\widetilde\ga_t}^\perp)^2k_{\widetilde\ga_t} - |k_{\widetilde\ga_t}|^2k_{\widetilde\ga_t}/2 + k_{\widetilde\ga_t}
$$
where $k_{\widetilde{\ga}_t}$ is the curvature of $\widetilde{\ga}_t$ (as the flow is invariant by translation and changing the parametrization of the evolving curves only affects the {\em tangential part} of the velocity). Since $\widetilde{\ga}_{t_\ep}$ is such reparametrization of $\overline{\ga}_{t_{j_\ep}} - p_{j_\ep}$ and this latter is close in $C^m(\SSS^1,\R^n)$ to $\ga_\infty$, possibly choosing smaller $\ep,\delta>0$ above, it easily follows that for every $t\in[t_{j_\varepsilon},t_{j_\ep} + \de)$ there holds
$$
\| Y_t \|_{H^4} <\si,
$$
where $\si>0$ is as in Proposition~\ref{prop:Loja} applied on $\ga_\infty$, and we possibly choose it smaller than the constant $\rho_\infty$.
 
We want now to prove that if $\ep>0$ is sufficiently small, then actually we can choose $\de=+\infty$ and $\| Y_t \|_{H^4} <\si$ for every time.

For $E$ as in Proposition~\ref{prop:Loja}, we have
\begin{align}
[\sE(\ga_t)-\sE(\ga_\infty)]^{1-\al}=&\,[\sE(\widetilde\ga_t)-\sE(\ga_\infty)]^{1-\al}\nonumber\\
=&\,\left[E(Y_t) - E(0) \right]^{1-\al}\nonumber\\
\le&\, C_1(\gamma_\infty,\sigma) \| \de E_{Y_t} \|_{(L^{2,\perp}_{\ga_\infty})^\star}\nonumber\\
=&\,C_1(\gamma_\infty,\sigma)\sup_{\Vert S\Vert_{L^{2,\perp}_{\ga_\infty}=1}}\int_{\SSS^1}  \left\langle |\widetilde\ga_t'|\bigl((\na_{\widetilde\ga_t}^\perp)^2k_{\widetilde\ga_t} + |k_{\widetilde\ga_t}|^2k_{\widetilde\ga_t}/2 - k_{\widetilde\ga_t}\bigr), S \right\rangle\,d\theta\nonumber\\
\leq&\,C_1(\gamma_\infty,\sigma)\sup_{\Vert S\Vert_{L^2(\SSS^1,\R^n)=1}}\int_{\SSS^1}  \left\langle |\widetilde\ga_t'|\bigl((\na_{\widetilde\ga_t}^\perp)^2k_{\widetilde\ga_t} + |k_{\widetilde\ga_t}|^2k_{\widetilde\ga_t}/2 - k_{\widetilde\ga_t}\bigr), S \right\rangle\,d\theta\nonumber\\
=&\,C_1(\gamma_\infty,\sigma)\left(\int_{\SSS^1}  |\widetilde\ga_t'|^2\bigl\vert(\na_{\widetilde\ga_t}^\perp)^2k_{\widetilde\ga_t} + |k_{\widetilde\ga_t}|^2k_{\widetilde\ga_t}/2 - k_{\widetilde\ga_t}\bigr\vert^2\,d\theta\right)^{1/2}\label{eqcar1}
\end{align}
where we can assume that $C_1(\gamma_\infty,\sigma)\geq 1$.\\
Now, $\scal{\widetilde{\ga}_t,\tau_{\ga_\infty}}=\scal{\ga_\infty,\tau_{\ga_\infty}}$ is time independent, then $\scal{\pa_t \widetilde{\ga},\tau_{\ga_\infty}}=0$ and possibly taking a smaller $\si>0$, we can suppose that $|\tau_{\ga_\infty}-\tau_{\widetilde{\ga}}|\le \tfrac12$ for any $t\ge t_{j_\ep}$ such that $\| Y_t \|_{H^4} <\si$. Hence,
\[
|(\pa_t\widetilde{\ga})^\perp|= | \pa_t \widetilde{\ga} - \scal{  \pa_t \widetilde{\ga} , \tau_{\widetilde{\ga}} } \tau_{\widetilde{\ga}} | 
=  | \pa_t \widetilde{\ga} + \scal{  \pa_t \widetilde{\ga} , \tau_{\ga_\infty} -  \tau_{\widetilde{\ga}} } \tau_{\widetilde{\ga}} |\ge |\pa_t\widetilde{\ga}| - |\pa_t\widetilde{\ga}||\tau_{\ga_\infty}-\tau_{\widetilde{\ga}}|
 \ge \frac12 |\pa_t\widetilde{\ga}|.
\]
Differentiating $H$, we then get
\begin{align}
\frac{d}{dt} H(t)=&\,\frac{d}{dt}[\sE(\widetilde\ga_t)-\sE(\ga_\infty)]^{\al}\\
=&\,\al H^{\frac{\al-1}{\al}}\delta \sE_{\widetilde\ga_t}(\partial_t\widetilde\ga)\\
=&\,-\al H^{\frac{\al-1}{\al}}\int_{\SSS^1}  |\widetilde\ga_t'|\bigl\vert(\na_{\widetilde\ga_t}^\perp)^2k_{\widetilde\ga_t} + |k_{\widetilde\ga_t}|^2k_{\widetilde\ga_t}/2 - k_{\widetilde\ga_t}\bigr\vert^2\,d\theta\nonumber\\
\leq&\,-\al H^{\frac{\al-1}{\al}}C_2(\ga_\infty,\sigma)\left(\int_{\SSS^1}  \bigl\vert(\partial_t\widetilde\ga)^{\perp}\bigr\vert^2\,d\theta\right)^{1/2}\left(\int_{\SSS^1}  |\widetilde\ga_t'|^2\bigl\vert(\na_{\widetilde\ga_t}^\perp)^2k_{\widetilde\ga_t} + |k_{\widetilde\ga_t}|^2k_{\widetilde\ga_t}/2 - k_{\widetilde\ga_t}\bigr\vert^2\,d\theta\right)^{1/2}\nonumber\\
\leq&\,-H^{\frac{\al-1}{\al}}C(\ga_\infty,\sigma) \|\pa_t\widetilde\ga \|_{L^2(d\theta)}[\sE(\widetilde\ga_t)-\sE(\widetilde\ga_\infty)]^{1-\al}\\ 
=&\,-C(\ga_\infty,\sigma)\|\pa_t\widetilde\ga \|_{L^2(d\theta)},\label{eqcar777}
\end{align}
where $C(\ga_\infty,\sigma)=\al C_2(\ga_\infty,\sigma)/2C_1(\ga_\infty,\sigma)$. This inequality clearly implies the estimate
\begin{equation}\label{eqcar778}
C(\ga_\infty,\sigma)\int_{\xi_1}^{\xi_2}\|\pa_t\widetilde\ga \|_{L^2(d\theta)}\,dt\leq H(\xi_1)-H(\xi_2)\leq H(\xi_1)
\end{equation}
for every $t_{j_\ep}\le \xi_1<\xi_2<t_{j_\ep}+\delta$ such that $\| Y_t \|_{H^4} <\si$. Hence, for such $\xi_1,\xi_2$ we have
\begin{align}
\Vert\widetilde\ga_{\xi_2}-\widetilde\ga_{\xi_1}\Vert_{L^2(d\theta)}
=&\,\left(\int_{\SSS^1} |\widetilde\ga_{\xi_2}(\theta)-\widetilde\ga_{\xi_1}(\theta)|^2\,d\theta\right)^{1/2}\\
\leq&\,\biggl(\int_{\SSS^1} \left(\int_{\xi_1}^{\xi_2}\partial_t\widetilde\ga(t,\theta)\,dt\,\right)^2d\theta\biggr)^{1/2}\\
=&\,\left\Vert\int_{\xi_1}^{\xi_2}\partial_t\widetilde\ga\,dt\,\right\Vert_{L^2(d\theta)}\\
\leq&\,\int_{\xi_1}^{\xi_2}\Vert\partial_t\widetilde\ga\Vert_{L^2(d\theta)}\,dt\\
\leq&\, \frac{{H(\xi_1)}}{C(\ga_\infty,\sigma)}\\
\le&\, \frac{\ep}{C(\ga_\infty,\sigma)},
\label{eq:CauchyL2}
\end{align}
where we used that $H(\xi_1)\le H(t_{j_\ep})\le \ep$ and the fact that $\bigl\Vert\int_{\xi_1}^{\xi_2} v\,dt\,\bigr\Vert_{L^2(d\theta)}\leq\int_{\xi_1}^{\xi_2}\Vert v\Vert_{L^2(d\theta)}\,dt$, holding for every smooth function $v:[{\xi_1},{\xi_2}]\times\SSS^1\to\R^n$, indeed
\begin{align*}
\left\Vert\int_{\xi_1}^{\xi_2} v\,dt\,\right\Vert_{L^2(d\theta)}^2
\leq&\,\int_{\SSS^1} \left(\int_{\xi_1}^{\xi_2}v(t,\theta)\,dt\,\right)^2d\theta\\
=&\,\int_{\SSS^1}\int_{\xi_1}^{\xi_2}v(t,\theta)\left(\int_{\xi_1}^{\xi_2}v(r,\theta)\,dr\,\right)dt\,d\theta\\
=&\,\int_{\xi_1}^{\xi_2}\int_{\SSS^1}v(t,\theta)\left(\int_{\xi_1}^{\xi_2}v(r,\theta)\,dr\,\right)d\theta\,dt\\
\leq&\,\int_{\xi_1}^{\xi_2}\Vert v\Vert_{L^2(d\theta)}\left\Vert\int_{\xi_1}^{\xi_2} v\,dt\,\right\Vert_{L^2(d\theta)}\,dt
\end{align*}
and such inequality follows.

Therefore, for $t\ge t_{j_\ep}$ such that $\| Y_t \|_{H^4} <\si$, we have
\[
\|Y_t\|_{L^2(d\theta)}=\| \widetilde{\ga}_t - \ga_\infty \|_{L^2(d\theta)}\le \| \widetilde{\ga}_t - \widetilde{\ga}_{t_{j_\ep}} \|_{L^2(d\theta)}  + 
\|   \widetilde{\ga}_{t_{j_\ep}}   - \ga_\infty  \|_{L^2(d\theta)}\le  \frac{\ep}{C(\ga_\infty,\sigma)} +\ep \sqrt{2\pi}.
\]
Then, by means of Gagliardo--Nirenberg interpolation inequalities (see~\cite{adams} or~\cite{Aubin}, for instance) and estimates~\eqref{eq:StimaParabolica2}, for every $l\geq 4$, we have
\[
\|Y_t \|_{H^l} \le C \|Y_t \|_{H^{l+1}}^a \|Y_t \|_{L^2(d\theta)}^{1-a} \le C(l,\ga_0,\ga_\infty,\sigma)\ep^{1-a},
\]
for some $a\in(0,1)$ and any $t\ge t_{j_\ep}$ such that $\| Y_t \|_{H^4} <\si$.\\
In particular setting $l+1=m\ge6$, if $\ep>0$ was chosen sufficiently small depending only on $\ga_0$, $\ga_\infty$ and $\sigma$, then $\|Y_t\|_{H^4}<\sigma/2$ for any time $t\ge t_{j_\ep}$, which means that we could have chosen $\delta=+\infty$ in the previous discussion.

Then, from estimate~\eqref{eq:CauchyL2} it follows that $\widetilde{\ga}_t$ is a Cauchy sequence in $L^2(d\theta)$ as $t\to+\infty$, therefore $\widetilde{\ga}_t $ converges in $L^2(d\theta)$ as $t\to+\infty$ to some limit curve $\widetilde{\gamma}_\infty$ (not necessarily coincident with $\gamma_\infty$). Moreover, by means of the above interpolation inequalities, repeating the argument for higher $m$ we see that such convergence is actually in $H^m$ for every $m\in\N$, hence in $C^m(\SSS^1,\R^n)$ for every $m\in\N$, by Sobolev embedding theorem. This implies that $\widetilde{\gamma}_\infty$ is a smooth critical point of $\sE$. As the original flow $\ga_t$ is a fixed translation of $\widetilde{\ga}_t$ (up to reparametrization) this finally completes the proof.
\end{proof}

\begin{remark}[Why is the \L ojasiewicz--Simon inequality necessary for the conclusion?]
We briefly discuss the reason why the \L ojasiewicz--Simon inequality gives a very strong improvement on the estimates and leads to the key inequality~\eqref{eq:CauchyL2}, which cannot be obtained simply by the standard variational evolution equation for the energy
$$
\frac{d\,}{dt}\sE(\gamma_t)=-\Vert\partial_t\gamma\Vert_{L^2(ds)}^2,
$$
that actually implies (in the notation and hypotheses of the previous proof)
$$
\frac{d}{dt}[\sE(\widetilde\ga_t)-\sE(\ga_\infty)]\leq-C(\ga_\infty,\sigma)\|\pa_t\widetilde\ga \|_{L^2(d\theta)}^2.
$$
This inequality is very similar to~\eqref{eqcar777} (if we choose $\al=1$ in the definition of $H$) and leads to the estimate
$$
C(\ga_\infty,\sigma)\int_{\xi_1}^{\xi_2}\|\pa_t\widetilde\ga \|_{L^2(d\theta)}^2\,dt\leq H(\xi_1)-H(\xi_2)\leq H(\xi_1)\leq\ep,
$$
which differs by~\eqref{eqcar778} {\em only for the exponent ``2'' on $\|\pa_t\widetilde\ga \|_{L^2(d\theta)}$ inside the integral}.\\
This makes a lot of difference, since in this case we are actually ``morally'' estimating the integral on an infinite time interval  
$$
C(\ga_\infty,\sigma)\int_{\xi_1}^{+\infty}\|\pa_t\widetilde\ga \|_{L^2(d\theta)}^2\,dt\leq\ep,
$$
that is the $L^2$--{\em in~time} norm of $\|\pa_t\widetilde\ga \|_{L^2(d\theta)}$, while in the above proof, by means of the \L ojasiewicz--Simon inequality, we got an estimate on the same function in the $L^1$--{\em in~time} norm. The estimate in the $L^1$--{\em in~time} norm is stronger because the time interval is unbounded and $\|\pa_t\widetilde\ga \|_{L^2(d\theta)}$ is uniformly bounded by \eqref{eq:StimaParabolica}. Therefore, trying to use such an $L^2$--in time estimate in order to get an inequality analogous to~\eqref{eq:CauchyL2}, that is, of kind
$$
\Vert\widetilde\ga_{\xi_2}-\widetilde\ga_{\xi_1}\Vert_{L^2(d\theta)}\le{C(\ga_\infty,\sigma)}\ep,
$$
clearly fails (anyway, such $L^2$--in time estimate is sufficient, and actually essential, in order to show the {\em sub--convergence} stated in Proposition~\ref{thm:SubConvergence}). So an $L^1$--{\em in~time} bound is absolutely needed and this is the reason of the key importance of the \L ojasiewicz--Simon inequality in showing the asymptotic {\em full convergence} of the flow.

We also notice that even assuming to have an inequality (in the notation and hypotheses of Proposition~\ref{prop:Loja}) as
\begin{equation*}
|\sE({\ga_\infty}+Y)-\sE({\ga_\infty})|\le C \| \delta E_Y \|_{(L^{2,\perp}_{\ga_\infty})^\star},
\end{equation*}
that corresponds to the case $\alpha=0$ in \eqref{eq:LojaFull}, this is not sufficient. Indeed, such an estimate is weaker than the  \L ojasiewicz--Simon inequality where $\al>0$, because we are interested in the case that the norm of $Y$ is small and so is the left hand side, and if we try to argue as in computations~\eqref{eqcar1} and~\eqref{eqcar777}, choosing any $\beta>0$ and setting $H=[\sE(\widetilde\ga_t)-\sE(\ga_\infty)]^{\beta}$, we obtain
\begin{align*}
\frac{d}{dt} H(t)=&\,\frac{d}{dt}[\sE(\widetilde\ga_t)-\sE(\ga_\infty)]^{\beta}\\
=&\,\beta H^{\frac{\beta-1}{\beta}}\delta \sE_{\widetilde\ga_t}(\partial_t\widetilde\ga)\\
\leq&\,-H^{\frac{\beta-1}{\beta}}C(\ga_\infty,\sigma) \|\pa_t\widetilde\ga \|_{L^2(d\theta)}[\sE(\widetilde\ga_t)-\sE(\widetilde\ga_\infty)]\\ 
=&\,-C(\ga_\infty,\sigma)H\|\pa_t\widetilde\ga \|_{L^2(d\theta)},
\end{align*}
and we get the (weak) estimate
\begin{equation}
C(\ga_\infty,\sigma)\int_{\xi_1}^{\xi_2}\|\pa_t\widetilde\ga \|_{L^2(d\theta)}\,dt\leq \log H(\xi_1)-\log H(\xi_2)
\end{equation}
which is, as before, clearly not sufficient to produce the necessary estimate on the $L^1$--{\em in~time} norm of $\|\pa_t\widetilde\ga \|_{L^2(d\theta)}$ on the time interval $[\xi_1,+\infty)$, as $\log H(\xi_2)\to-\infty$ as $\xi_2\to+\infty$. 
\end{remark}

\bibliographystyle{amsplain}
\bibliography{biblio}

\end{document}